\documentclass[preprint,12pt]{elsarticle}
\usepackage[top=20mm, bottom=20mm, left=20mm, right=20mm]{geometry}
\usepackage{amsthm,amsfonts,amssymb,amscd,amsmath,enumerate,verbatim,calc,graphicx,geometry}
\allowdisplaybreaks[1]
\usepackage{color}
\usepackage[all]{xy}
\newtheorem{theorem}{Theorem}[section]
\newtheorem{lemma}[theorem]{Lemma}

\newtheorem{corollary}[theorem]{Corollary}
\theoremstyle{definition}

\newtheorem{recall}[theorem]{Recall}
\newtheorem{example}[theorem]{Example}

\newtheorem{note}[theorem]{Note}

\journal{ }

\begin{document}
\begin{frontmatter}



\title{On a Van Kampen Theorem for Hawaiian Groups}


\author[]{Ameneh~Babaee}
\ead{Am.Babaee@um.ac.ir}
\author[]{Behrooz~Mashayekhy\corref{cor1}}
\ead{bmashf@um.ac.ir}
\author[]{Hanieh~Mirebrahimi}
\ead{h\_mirebrahimi@um.ac.ir}
\author[]{Hamid~Torabi}
\ead{h.torabi@ferdowsi.um.ac.ir}

\address{Department of Pure Mathematics, Center of Excellence in Analysis on Algebraic Structures, Ferdowsi University of
Mashhad,\\
P.O.Box 1159-91775, Mashhad, Iran.}
\cortext[cor1]{Corresponding author}
\begin{abstract}
 The paper is devoted to study the $n$th Hawaiian group $\mathcal{H}_n$, $n \ge 1$, of the wedge sum of two spaces $(X,x_*) = (X_1, x_1) \vee (X_2, x_2)$.
 Indeed, we are going to give some versions of the van Kampen theorem for Hawaiian groups of the wedge sum of spaces.
 First, among some results on Hawaiian groups of semilocally strongly contractible spaces, we present a structure for the $n$th Hawaiian group of the wedge sum of CW-complexes.
 Second, we give more informative structures for the $n$th Hawaiian group of the wedge sum $X$, when $X$ is semilocally $n$-simply connected at $x_*$.
 Finally, as a consequence, by generalizing the well-known Griffiths space for dimension $n\geq 1$, we give some information about the structure of Hawaiian groups of Griffiths spaces at any points.
\end{abstract}


\end{frontmatter}




\section{Introduction and Motivation}
In 2000, K. Eda and K. Kawamura \cite{EdaKaw} generalized the well-known Hawaiian earring space to higher dimensions $n\in \mathbb{N}$ as the following
subspace of $(n+1)$-dimensional Euclidean space ${\mathbb{R}}^{(n+1)}$
\[
\mathbb{HE}^n = \{(r_0,r_1,...,r_n)\in {\mathbb{R}}^{(n+1)}\ |\ (r_0-1/k)^2+{\sum}_{i=1}^n r_i^2 = (1/k)^2, k\in {\mathbb{N}}\}.
\]
Here $\theta = (0,0,...,0)$ is regarded as the base point of the $n$-dimensional Hawaiian earring $\mathbb{HE}^n$, and $\mathbb{S}_k^n$ denotes the $n$-sphere in $\mathbb{HE}^n$ with radius 1/k.

%
In 2006, U.H. Karimov and D. Repov\v{s} \cite{KarRep} defined a new notion, the  $n$th Hawaiian group of
a pointed space $(X,x_0)$, denoted by ${\mathcal{H}}_n(X,x_0)$, to be the set of all pointed  homotopy classes $[f]$, where
$f:({\mathbb{HE}^n},\theta)\to  (X,x_0)$ is a pointed map. The operation of the $n$th Hawaiian group comes naturally from the operation of the $n$th homotopy group so that the following map
\begin{equation}
{\varphi}:{\mathcal{H}}_n(X,x_0) \to    \prod_{\mathbb{N}}{\pi}_n(X,x_0),  \tag{I}
\end{equation}
defined by $\varphi([f])=([f{\mid}_{\mathbb{S}_1^n}], [f{\mid}_{\mathbb{S}_2^n}],... )$ is a homomorphism, for all $n \in \mathbb{N}$. For every pointed space $(X,x_0)$, the image of $\varphi$ and also $\mathcal{H}_n(X,x_0)$ contain $\prod^W_{\mathbb{N}}{\pi}_n(X,x_0)$ as a normal subgroup (see the proof of \cite[Theorem 2.13]{1}).

One can see that ${\mathcal{H}}_n:hTop_{\ast} \to   Groups $ is a covariant functor from the pointed homotopy category,
$hTop_{\ast}$, to the category of all groups, $Groups$, for any $n\geq 1$. If
$\beta:(X,x_0)\to   (Y,y_0)$ is a pointed map, then
$\mathcal{H}_n(\beta)={\beta}_{\ast}:{\mathcal{H}}_n (X,x_0)\to   {\mathcal{H}}_n (Y,y_0)$ defined by
${\beta}_{\ast}([f])= [\beta \circ f]$ is a homomorphism (see \cite{KarRep}).

Although the $n$th Hawaiian group functor is a pointed homotopy invariant functor on the category of all pointed topological spaces, it is not freely homotopy invariant. Because unlike other homotopy invariant functors, Hawaiian groups of contractible spaces are not necessarily trivial.
Karimov and Repov\v{s} \cite{KarRep} gave a contractible space, the cone over $\mathbb{HE}^1$, with nontrivial
$1$th Hawaiian group at some points (consisting of the points at which $C\mathbb{HE}^1$ is not locally $1$-simply connected), but with trivial homotopy, homology and cohomology groups. More precisely, it can be shown that
${\mathcal{H}}_1(C({\mathbb{HE}}^1),\theta)$ is uncountable, using  \cite[Theorem 2]{KarRep}: ``{\it if $X$ is first countable at $x_0$, then countability of $n$th Hawaiian group ${\mathcal{H}}_n (X,x_0)$ implies locally $n$-simply connectedness of $X$ at $x_0$}."
Furthermore, a converse of the above statement can be found in  \cite[Corollaries 2.16 and 2.17]{1}: ``{\it let $X$ be first countable at $x_0$. Then $\mathcal{H}_n(CX, \tilde{x})$ is trivial if and only if $X$ is locally $n$-simply connected at $x_0$ and  it is uncountable otherwise}".
Accordingly, this functor can help us to obtain some local properties of spaces.
In addition, unlike homotopy groups, Hawaiian groups of pointed space $(X,x_0)$ depend on the behaviour of $X$ at $x_0$, and then their structures depend on the choice of the base point. In this regard, there exist some examples of path connected spaces with non-isomorphic Hawaiian groups at different points, such as the $n$-dimensional Hawaiian earring, where $n \ge 2$ (see \cite[Corollary 2.11]{1}).

Despite the above different behaviors between Hawaiian groups and homotopy groups, they have some similar behaviors. For instance, it was proved that similar to the $n$th homotopy group, the $n$th Hawaiian group of any pointed space is abelian,  for $n \ge 2$ \cite[Theorem 2.3]{1}. Also, the Hawaiian groups preserve products in the category $hTop_*$ \cite[theorem 2.12]{1}.

In this paper, we investigate the Hawaiian groups of the coproduct in the category $hTop_*$ which is the wedge sum of a given family of pointed spaces. In fact, we are going to give some versions of the van Kampen theorem for Hawaiian groups of the wedge sum of spaces. In Section 2, among some results on Hawaiian groups of semilocally strongly contractible spaces, we intend to present a structure for the $n$th Hawaiian group of the wedge sum of CW-complexes. A space $X$ is called  semilocally strongly contractible at $x_0$ if there exists some open neighbourhood $U$ of $x_0$ such that the inclusion $i :U \hookrightarrow X$ is nullhomotopic in $X$ relative to $\{x_0\}$, in other words, $i : (U,x_0) \hookrightarrow (X,x_0)$ is nullhomotopic (see \cite{EdaKaw}).

In Section 3, we present the $n$th Hawaiian group of the wedge sum $(X,x_*) = (X_1, x_1) \vee (X_2, x_2)$ as the semidirect product of two its subgroups that are more perceptible, when  $X$ is  semilocally $n$-simply connected at $x_*$. Also, we prove that the Hawaiian group of a pointed space equals the Hawaiian group of every neighbourhood of the base point if all $n$-loops are small.  An $n$-loop $\alpha: (\mathbb{S}^n, 1) \to (X, x)$ is called small, if for each neighbourhood $U$ of $x$, $\alpha$ has a homotopic representative in $U$ (see \cite{PasGha, Vir}).

In Section 4, generalizing the well-known Griffiths space, we define the $n$th Griffiths space for $n \ge 2$,  as the wedge sum of two copies of the cone over the $n$-dimensional Hawaiian earring. For the sake of clarity, we call the well known Griffiths space as the $1$th Griffiths space. Then, using results of Sections 2 and 3, we intend to give some information about the structure of Hawaiian groups of Griffiths spaces at any points.

In this paper, all homotopies are relative to the base point, unless stated otherwise. 

\section{Hawaiian Groups of Semilocally Strongly Contractible Spaces}

In this section, we investigate Hawaiian groups of wedge sum of pointed spaces which are semilocally strongly contractible at the base points. The property semilocally strongly contractibility was defined by Eda and Kawamura \cite{EdaKaw}. First, we compare semilocally strongly contractible property with some familiar properties, such as locally contractible, locally strongly contractible and semilocally contractible properties.
\begin{recall}
Let $(X, x)$ be a pointed space, then $X$ is called locally contractible at $x$ if for each open neighbourhood $U$ of $x$ in $X$, there exists some open neighbourhood $V$ of $x$ contained in $U$ such that the inclusion $V \hookrightarrow U$ is freely nullhomotopic.

We say that $X$ is locally strongly contractible at $x$ if for each open neighbourhood $U$ of $x$ in $X$, there exists some open neighbourhood $V$ of $x$ contained in $U$ such that the inclusion $V \hookrightarrow U$ is nullhomotopic relative to $\{x\}$, or briefly, $(V, x) \hookrightarrow (U, x)$ is nullhomotopic.

Moreover, $X$ is called semilocally contractible at $x$, if there exists some open neighbourhood $U$ of $x$ such that the inclusion $i :U \hookrightarrow X$ is freely nullhomotopic in $X$.
\end{recall}

One can see that there are some relations between these properties.
Locally strongly contractible property implies locally contractibility, and locally contractible property implies semilocally contractibility, but the converse statements do not necessarily hold. Moreover, locally strongly contractible property implies semilocally strongly contractibility, and semilocally strongly  contractible implies semilocally contractibility, but converse statements do not necessarily hold. The following example shows that  a given space can behave differently at different points.

\begin{example}
The $1$th Griffiths space $\mathcal{G}_1$ is locally strongly contractible at two vertices $v$ and $v'$. Therefore, it is locally contractible, semilocally strongly contractible  and semilocally contractible at two vertices $v$ and $v'$. Moreover, $\mathcal{G}_1$ is semilocally contractible at all points except at the common point $x_*$. Also, $\mathcal{G}_1$ is semilocally contractible at any point of $A$ and $A'$, but not semilocally strongly contractible. Finally, $\mathcal{G}_1$ is neither semilocally strongly contractible, semilocally contractible nor even semilocally $1$-simply connected at $x_*$ (see \cite{Gri}).

\begin{figure}[!ht]
\centering
\includegraphics[scale=.5]{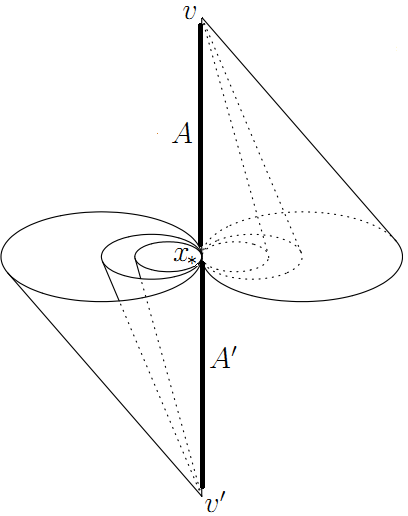}
\caption{The $1$th Griffiths space.}\label{fi1}
\end{figure}
\end{example}

The $1$-dimensional Hawaiian earring $\mathbb{HE}^1$ is not semilocally contractible at the origin. However, it is locally strongly contractible at other points.
There exist many spaces which are semilocally contractible at each point, but not strongly at some points.
For example, consider $C\Delta$ as the cone over the Seirpeinski gasket. It is semilocally contractible at any points, but it is semilocally strongly contractible just at the vertex.

The following lemma describes the Hawaiian group of the wedge sum of two spaces by its subgroups. This construction helps us to obtain next isomorphisms.

\begin{lemma}\label{le4.1}
Let $(X_1,x_1)$ and $(X_2, x_2)$ be two pointed spaces, and $(X, x_*) = (X_1, x_1) \vee (X_2, x_2)$.
\begin{enumerate}[(i)]
\item
If $U_1$ and $U_2$ are two open neighbourhoods of $x_1$ and $x_2$ in $X_1$ and $X_2$, respectively, and $j:U_1 \vee U_2 \to X$ is the inclusion map, then
\begin{equation}\label{eq2.1nn}
\mathcal{H}_n(X, x_*) = j_* \mathcal{H}_n( U_1 \vee U_2, x_*) {\prod_{\mathbb{N}}}^W \pi_n(X,x_*).
\end{equation}
\item
If $x \in X_1 \setminus \{x_1\}$ and $\{x_1\}$ is closed in $X_1$, then
\begin{equation}\label{eq2.2nn}
\mathcal{H}_n(X, x) =
\mathcal{H}_n(X_1, x)   {\prod_{\mathbb{N}}}^W \pi_n(X,x).
\end{equation}
\end{enumerate}
\end{lemma}

\begin{proof}
\begin{enumerate}[(i)]
\item
Let $U_1$ and $U_2$ be two open neighbourhoods of $x_1$ and $x_2$ in $X_1$ and $X_2$, respectively, $j: U_1 \vee U_2 \to X$ be the inclusion map, and
$f: (\mathbb{HE}^n, \theta) \to (X,x_*)$ be a pointed map. Since $U_1 \vee U_2$ is open in $X$,
there exists $K \in \mathbb{N}$ such that  if $k \geq K$, then $im(f|_{\widetilde{\bigvee}_{k \geq K}\mathbb{S}_k^n}) \subseteq U_1 \vee U_2$. We can define $\underline{f}, \overline{f}: \mathbb{HE}^n \to (X,x)$ by
$\underline{f}|_{\bigvee_{k <K}\mathbb{S}_k^n} = f|_{\bigvee_{k <K}\mathbb{S}_k^n}$,
$\underline{f}|_{\widetilde{\bigvee}_{k \geq K}\mathbb{S}_k^n} = C|_{\widetilde\bigvee_{k \ge K}\mathbb{S}_k^n}$, $\overline{f}|_{\bigvee_{k <K}\mathbb{S}_k^n} = C|_{\bigvee_{k <K}\mathbb{S}_k^n}$, and
$\overline{f}|_{\widetilde{\bigvee}_{k \geq K}\mathbb{S}_k^n} = f|_{\widetilde\bigvee_{k \ge K}\mathbb{S}_k^n}$, where $C$ is the constant map.
Obviously, $[f] = [\underline{f}][\overline{f}] = [\overline{f}][\underline{f}]$.  Moreover, $[\underline{f}]$ is an element of $\prod_{\mathbb{N}}^W \pi_n(X,x_*)$ and also, $[\overline{f}]$ is an element of $j_* \mathcal{H}_n(U_1 \vee U_2, x_*)$. Therefore, $\mathcal{H}_n(X, x_*)$ is generated by $j_*\mathcal{H}_n(U_1 \vee U_2, x_*) \cup  \prod_{\mathbb{N}}^W \pi_n(X,x_*)$.
The equality \eqref{eq2.1nn} holds because $\prod_{\mathbb{N}}^W \pi_n(X, x_*) \unlhd \mathcal{H}_n(X,x_*)$ (see the proof of \cite[Theorem 2.13]{1}).

 \item
  Since $X_1$ is a retract of $X$, one can consider $\mathcal{H}_n(X_1, x)$ as a subgroup of $\mathcal{H}_n(X, x) $.
Let $f: (\mathbb{HE}^n, \theta) \to (X,x)$ be a pointed map. Since $\{x_1\}$ is closed in $X_1$, $X_1 \setminus \{x_1\}$ is open. Let $i: X_1 \setminus \{x_1\} \to X$ be the inclusion map. Similar to the previous part, $[f]$ can be factorized as
 $[f] = [\underline{f}][\overline{f}] = [\overline{f}][\underline{f}]$, where $[\underline{f}]$ is an element of $\prod_{\mathbb{N}}^W \pi_n(X,x)$ and also $[\overline{f}]$ is an element of $i_*\mathcal{H}_n(X_1 \setminus \{x_1\}, x)$ which is a subgroup of $\mathcal{H}_n(X_1, x)$. Thus, $\mathcal{H}_n(X, x)$ is generated by $\mathcal{H}_n(X_1, x) \cup  \prod_{\mathbb{N}}^W \pi_n(X,x)$. Again, equality \eqref{eq2.2nn} holds by normality of $\prod_{\mathbb{N}}^W \pi_n(X,x)$.
\end{enumerate}
\end{proof}
Note that equalities \eqref{eq2.1nn} and \eqref{eq2.2nn} may not be the same. For instance, in Example \ref{ex3.9n} one group is trivial, but not the other one, in general.

\begin{note}\label{le3.2n}
Using the property that every open neighbourhood in the wedge sum is a wedge sum of open neighbourhoods, one can rewrite the proof of Lemma \ref{le4.1} for arbitrary spaces stated as follows.
Let $(X,x_0)$ be a pointed space. Then
$\mathcal{H}_n(X,x_0) = j_* \mathcal{H}_n(U, x_0) \prod_{\mathbb{N}}^W \pi_n(X,x_0)$,
for each open neighbourhood $U$ with $x_0 \in U \subseteq X$ and $j : U \to X$ as the inclusion map.
\end{note}

A result similar to the following theorem was proved in \cite[Theorem 2.5]{1} by a slightly different argument.
\begin{theorem}\label{th2.6n}
Let  $(X_1, x_1)$ and $(X_2, x_2)$ be two pointed spaces, $(X, x_*) = (X_1, x_1)  \vee (X_2, x_2)$, and $n \ge 1$.
If $X_1$ and $X_2$ are semilocally strongly contractible at $x_1$ and $x_2$, respectively, then
\begin{equation*}\label{eq2.3nn}
\mathcal{H}_n(X, x_*) = {\prod_{\mathbb{N}}}^W \pi_n(X, x_*).
\end{equation*}
\end{theorem}
\begin{proof}
Since $X_1$ and $X_2$ are semilocally strongly contractible at $x_1$ and $x_2$, respectively, there exist open neighbourhoods $U_1$ of $x_1$ and $U_2$ of $x_2$ such that inclusion maps $(U_1, x_1) \hookrightarrow (X_1, x_1)$ and $(U_2, x_2) \hookrightarrow (X_2, x_2)$ are nullhomotopic. By joining these homotopies, one can see that $j:U_1 \vee U_2 \to X_1 \vee X_2$ is nullhomotopic and hence  $j_* \mathcal{H}_n(U_1 \vee U_2, x_*)$ is trivial. The result holds by Lemma \ref{le4.1}.
\end{proof}

The following example shows that Theorem \ref{th2.6n} does not hold without condition semilocally strongly contractible on both of spaces.
\begin{example}\label{ex2.10nn}
Consider $x_*$ as the common point of the $1$th Griffiths space. We show that if Theorem \ref{th2.6n} holds for the $1$th Griffiths space at $x_*$, then $\pi_1(\mathcal{G}_1, x_*)$ is trivial which is a contradiction (see \cite{Gri}). Assume that $\mathcal{H}_1(\mathcal{G}_1, x_*) = \prod_{\mathbb{N}}^W \pi_1(\mathcal{G}_1, x_*)$, then the homomorphism $\varphi: {\mathcal{H}}_n(X,x_0) \to    \prod_{\mathbb{N}}{\pi}_n(X,x_0)$ (see (I)) can be considered as the natural injection. If we show that the homomorphism $\varphi$ is surjective, then the natural injection $\prod^W_{\mathbb{N}} \pi_1(\mathcal{G}_1, x_*) \to \prod_{\mathbb{N}}  \pi_1(\mathcal{G}_1, x_*)$ is surjective which is impossible in a nontrivial way.

To prove surjectivity of $\varphi$, let $\{[f_k]\} \in \prod_{\mathbb{N}} \pi_1(\mathcal{G}_1, x_*)$. Since any $n$-loop at $x_*$ is small, we can find a homotopic representative $f'_k$ of $[f_k]$ in $U_k$, for all $k \in \mathbb{N}$. Now we can define $f:\mathbb{HE}^1 \to \mathcal{G}_1$ by $f|_{\mathbb{S}_k^1} = f'_k$, satisfying $\varphi ([f])= \{[f_k]\}$.
 \end{example}

CW-complexes  are locally homeomorphic to some cells. Thus, they are semilocally strongly contractible at any point. The following result presents an isomorphism for the Hawaiian group of the wedge sum of CW-complexes.

\begin{corollary}\label{co2.5n}
Let $X_1$ and $X_2$ be two locally finite $(n-1)$-connected CW-complexes, $(X, x_*)=(X_1, x_1) \vee (X_2, x_2)$, and $n \geq 2$. Then
\begin{equation}\label{eq2.1}
\mathcal{H}_n(X,x_*) \cong\mathcal{H}_n(X_1, x_1) \oplus \mathcal{H}_n(X_2, x_2).
\end{equation}
\end{corollary}

\begin{proof}
CW-complexes are semilocally strongly contractible. Thus by Theorem \ref{th2.6n},
$\mathcal{H}_n(X,x) \cong \bigoplus_{\mathbb{N}} \pi_n(X,x)$, when  $n \geq 2$.
Now by \cite[Proposition 6.36]{Swi},
$\pi_n(X,x) \cong \pi_n(X_1) \oplus  \pi_n(X_2)$, and after a rearrangement,
$\mathcal{H}_n(X,x) \cong \bigoplus_{\mathbb{N}} \pi_n(X_1) \oplus \bigoplus_{\mathbb{N}}\pi_n(X_2)$. We obtain the result, using \cite[Theorem 2.5]{1}.
\end{proof}

An analogous isomorphism for \eqref{eq2.1} does not hold, when $n =1$. To obtain such an isomorphism on the $1$th Hawaiian group, we must replace direct sum by free product, because the fundamental group and the $1$th Hawaiian group are not abelian groups, in general. Using Theorem \ref{th2.6n} and the van Kampen Theorem for wedge sum we have the following result.
\begin{corollary}\label{co2.10n}
Let $X_1$ and $X_2$ be two semilocally strongly contractible spaces at $x_1$ and $x_2$, respectively, and $(X, x_*)= (X_1, x_1) \vee (X_2, x_2)$. Then
\begin{equation}\label{eq2.2n}
\mathcal{H}_1(X, x_*) \cong {\prod_{\mathbb{N}}}^W \big( \pi_1(X_1,x_1) * \pi_1(X_2,x_2) \big).
\end{equation}
\end{corollary}

Note that the isomorphism \eqref{eq2.2n} is not similar to the case $n \ge 2$, even if $X$ is a special CW-complex.
Because if we consider $\mathcal{H}_1(X, x_*) \cong\mathcal{H}_1(X_1,x_1) * \mathcal{H}_1(X_2,x_2)$, then by Theorem \ref{th2.6n}, $\mathcal{H}_1(X, x_*) \cong \prod^W_{\mathbb{N}} \pi_1(X_1,x_1) * \prod^W_{\mathbb{N}}  \pi_1(X_2,x_2)$. Hence, by isomorphism \eqref{eq2.2n} we must have
\[
{\prod_{\mathbb{N}}}^W \big( \pi_1(X_1,x_1) * \pi_1(X_2,x_2) \big) \cong {\prod_{\mathbb{N}}}^W \pi_1(X_1,x_1) * {\prod_{\mathbb{N}}}^W  \pi_1(X_2,x_2),
\]
which is impossible by \cite[Page 183, 6.3.10]{Rob}, in a nontrivial way.
Therefore, $\mathcal{H}_1(X, x_*)\not \cong \mathcal{H}_1(X_1,x_1) * \mathcal{H}_1(X_2,x_2)$, in a nontrivial way.

\section{Hawaiian Groups of Semilocally $n$-Simply Connected Spaces}

In this section, we study more on Hawaiian groups of the wedge sum in semilocally $n$-simply connected spaces. We present results for $n=1$ and $n \ge 2$, separately, due to the difference in group structures.

Recall that for $n \geq 1$, a space $X$ is called $n$-simply connected at $x$ if $\pi_n(X,x)$ is trivial and it is called $n$-connected at $x$ if $\pi_j(X,x)$ is trivial, for $ 1 \leq j \leq n$.
Also, $X$ is called semilocally $n$-simply connected at $x$ if there exists a neighbourhood $U$ of $x$ such that the homomorphism $\pi_n (j): \pi_n(U,x) \to \pi_n(X,x)$, induced by the inclusion, is trivial.
\begin{theorem}\label{th2.8n}
Let $(X_1,x_1)$ and $(X_2, x_2)$ be two pointed spaces  and $(X, x_*) = (X_1, x_1) \vee (X_2, x_2)$. If $X$ is semilocally $1$-simply connected at $x_*$, then
\[
\mathcal{H}_1(X, x_*) \cong j_* \mathcal{H}_1 (U_1 \vee U_2) \ltimes {\prod_{\mathbb{N}}}^W  \pi_1(X, x_*),
\]
for some neighbourhood $U_1 \vee U_2$ of $x_*$ with the inclusion map $j: U_1\vee U_2 \to X$.
\end{theorem}

\begin{proof}
By Lemma \ref{le4.1} part 1, $\mathcal{H}_1(X, x_*) = j_*\mathcal{H}_1(U_1 \vee U_2, x_*)  \prod_{\mathbb{N}}^W \pi_1(X,x_*)$ for each open neighbourhood $U_1 \vee U_2$ of $x_*$. Let $U_1 \vee U_2$ be the neighbourhood for which $\pi_1(j):\pi_1(U_1 \vee U_2, x_*) \to \pi_1(X, x_*)$ is the trivial homomorphism and let $[f] \in j_*\mathcal{H}_1(U_1 \vee U_2, x_*) \cap \prod_{\mathbb{N}}^W \pi_1(X,x_*)$. Since $[f] \in j_*\mathcal{H}_1(U_1 \vee U_2, x_*)$, there exists $\tilde{f} :(\mathbb{HE}^1, \theta )\to (U_1  \vee U_2,x_*)$ such that $j_* [\tilde{f}] = [f]$, or equivalently, $j \circ \tilde{f} \simeq f$. Also, since $[f] \in \prod_{\mathbb{N}}^W \pi_1(X,x_*)$, $f$ can be assumed as a map with $f|_{\widetilde{\bigvee}_{k \geq K} \mathbb{S}_k^1} = C|_{\widetilde{\bigvee}_{k \geq K} \mathbb{S}_k^1}$ for some $K \in \mathbb{N}$. Hence $j \circ \tilde{f}|_{\widetilde{\bigvee}_{k \geq K} \mathbb{S}_k^1} \simeq C|_{\widetilde{\bigvee}_{k \geq K} \mathbb{S}_k^1}$. Using \cite[Lemma 2.2]{1}, one can replace $\tilde{f}$ with a map $\hat{f}$ such that $\hat{f}|_{\widetilde{\bigvee}_{k \geq K} \mathbb{S}_k^1} \simeq C|_{\widetilde{\bigvee}_{k \geq K} \mathbb{S}_k^1}$. Thus $[\hat{f}] \in \prod_{\mathbb{N}}^W \pi_1(U_1 \vee U_2,x_*)$ and then $[f] \in \prod_{\mathbb{N}}^W \pi_1(j)\pi_1(U_1 \vee U_2,x_*)$ which is trivial, because of the choice of $U_1 \vee U_2$.
Hence $j_*\mathcal{H}_1(U_1 \vee U_2, x_*) \cap \prod_{\mathbb{N}}^W \pi_1(X,x_*) = \langle e \rangle$, and the result holds by normality of $\prod_{\mathbb{N}}^W \pi_1(X,x_*)$ in $\mathcal{H}_1(X, x_*)$.
\end{proof}

By a similar argument, we can conclude the following result for the wedge sum of $1$-simply connected spaces.

\begin{theorem}\label{th4.2n}
Let $(X_1,x_1)$ and $(X_2, x_2)$ be two pointed spaces and $(X, x_*) = (X_1, x_1) \vee (X_2, x_2)$.  If   $\pi_1(X_1, x) = \langle e \rangle$ for some $x \in X_1 \setminus \{x_1\}$, then
\[
\mathcal{H}_1(X, x) \cong
\mathcal{H}_1(X_1, x) \ltimes {\prod_{\mathbb{N}}}^W  \pi_1(X,x).
\]
\end{theorem}

In the following two theorems, we reconstruct isomorphisms in Theorems \ref{th2.8n} and \ref{th4.2n} for $n \ge 2$. In this case, since all groups are abelian and  all subgroups are normal, semidirect product $\ltimes$ must be replaced by direct sum $\oplus$.

\begin{theorem}
Let $(X_1,x_1)$ and $(X_2, x_2)$ be two pointed spaces, $(X, x_*) = (X_1, x_1) \vee (X_2, x_2)$, and $n \ge 2$.  If  $X$ is semilocally $n$-simply connected at $x_*$, then
\[
\mathcal{H}_n(X,x_*) \cong j_* \mathcal{H}_n (U_1 \vee U_2) \oplus \bigoplus_{\mathbb{N}}  \pi_n(X,x_*),
\]
for some neighbourhood $U_1 \vee U_2$ of $x_*$ with the inclusion map $j : U_1 \vee U_2 \to X$.
\end{theorem}

%

\begin{theorem}\label{th4.4}
Let $(X_1,x_1)$ and $(X_2, x_2)$ be two pointed spaces,  $(X, x_*) = (X_1, x_1) \vee (X_2, x_2)$, and $n \geq 2$. If $\pi_n(X_1, x) = \langle e \rangle$ for some $x_1 \neq x \in X_1$, then
\[
\mathcal{H}_n(X, x) \cong
\mathcal{H}_n(X_1, x) \oplus \bigoplus_{\mathbb{N}} \pi_n(X,x) .
\]
\end{theorem}


\v{Z}. Virk \cite{Vir} defined small $1$-loop and studied  small loop spaces. Note that a nullhomotopic loop is a small loop.
H. Passandideh and F.H. Ghane \cite{PasGha}  defined and studied the notions of $n$-homotopically Hausdorffness and small $n$-loops, for $n \ge 2$.  An $n$-loop $\alpha: (\mathbb{S}^n, 1) \to (X, x)$ is called small if it has a homotopic representative in every open neighbourhood of $x$.

\begin{theorem}\label{th2.5n}
Let $(X_1,x_1)$ and $(X_2, x_2)$ be two pointed spaces, $(X, x_*) = (X_1, x_1) \vee (X_2, x_2)$, and $n \ge 1$. Also, let $\{x_1\}$ be a  closed subset in $X_1$, and $x_1 \neq x\in X_1$.
\begin{enumerate}[(i)]
\item
 All $n$-loops in $X$ at $x_*$ are small if and only if
\begin{equation}\label{eq3.6.1}
\mathcal{H}_n(X, x_*) = j_* \mathcal{H}_n( U_1 \vee U_2, x_*),
\end{equation}
for any neighbourhoods $U_1$ and $U_2$ of $x_1$ and $x_2$ in $X_1$ and $X_2$, respectively, when $j:U_1 \vee U_2 \to X$ is the inclusion map.

\item
 If all $n$-loops in $X$ at $x$ are small, then
\begin{equation}\label{eq3.6.2}
\mathcal{H}_n(X, x) =
\mathcal{H}_n(X_1, x).
\end{equation}
\end{enumerate}
\end{theorem}

\begin{proof}
\begin{enumerate}[(i)]
\item
Let $U_1$ and $U_2$ be two arbitrary neighbourhoods of $x_1$ and $x_2$, respectively, and let $[f] \in \prod_{\mathbb{N}}^W \pi_n(X,x_*)$.  We show that $[f] \in  j_* \mathcal{H}_n( U_1 \vee U_2, x_*)$, and then the equality \eqref{eq3.6.1} is obtained by Lemma \ref{le4.1}.
Since $[f] \in \prod_{\mathbb{N}}^W \pi_n(X,x_*)$, one can consider $f$ as $f|_{\widetilde{\bigvee}_{k \geq K} \mathbb{S}_k^n} = C|_{\widetilde{\bigvee}_{k \geq K} \mathbb{S}_k^n}$ for some $K \in \mathbb{N}$. Moreover, since each $n$-loop in $X$ is small at $x_*$, any $n$-loop $\alpha$ is homotopic to some $n$-loop in $U_1 \vee U_2$, say $\tilde{\alpha}:\mathbb{S}^n \to U_1 \vee U_2$, such that $j \circ \tilde{\alpha} \simeq \alpha$. By induction on finite join of $n$-loops,
 $f|_{\widetilde{\bigvee}_{k < K} \mathbb{S}_k^n}$ has a homotopic representative $\tilde{f}$ in $U_1 \vee U_2$ with $j \circ \tilde{f}  \simeq f|_{\widetilde{\bigvee}_{k < K}\mathbb{S}_k^n}$. Therefore,  $[f] \in  j_* \mathcal{H}_n( U_1 \vee U_2, x_*)$.

Conversely, let $\alpha$ be an $n$-loop in $X$ at $x_*$. Consider the map $f:(\mathbb{HE}^n, \theta) \to (X,x_*)$ so that $f|_{\mathbb{S}_1^n} = \alpha$ and $f|_{\mathbb{S}_k^n} = c$ for $k >1$. Then $[f] \in j_*\mathcal{H}_n(U_1\vee U_2, x_*)$ by equality \eqref{eq3.6.1} for any neighbourhoods $U_1$ and $U_2$. Let $[\tilde{f}]$ be the element of $\mathcal{H}_n(U_1\vee U_2, x_*)$ such that $j \circ \tilde{f} = f$. Hence, $j \circ \tilde{f}|_{\mathbb{S}_1^n} = f|_{\mathbb{S}_1^n} = \alpha$, that is $\alpha$ is homotopic to some $n$-loop in $U_1 \vee U_2$. Since $U_1$ and $U_2$ are arbitrary neighbourhoods, $\alpha$ is a small $n$-loop.

\item
Since $X_1 \setminus \{x_1\}$ is open and all $n$-loops at $x$ in $X$  are small, similar to the proof of the previous part, one can show that $\prod_{\mathbb{N}}^W \pi_n(X,x) \subseteq i_*\mathcal{H}_n(X_1 \setminus \{x_1\}, x)$, where $i: X_1 \setminus \{x_1\} \to X$ is the inclusion map. Moreover, $i_*\mathcal{H}_n(X_1 \setminus \{x_1\}, x)$ is contained in $\mathcal{H}_n(X_1, x)$ as a subgroup. Thus, by Lemma \ref{le4.1}, the equality \eqref{eq3.6.2} holds.
\end{enumerate}
\end{proof}

Since a nullhomotopic $n$-loop is a special case of small $n$-loop, the equalities \eqref{eq3.6.1} and \eqref{eq3.6.2} hold for $n$-simply connected spaces.
%
%
%
%
%
For example, Theorem \ref{th2.5n} holds for two cones joining at their vertices which is not only $n$-simply connected, but also contractible. Recall that the $1$th Griffiths space, the wedge sum of two cones, is not contractible, even more Griffiths \cite{Gri} proved that it is not $1$-simply connected.

\begin{example}\label{ex3.9n}
For a space $X$, put $Y= \frac{X\times [-1,1]}{X \times \{0\}}$  the wedge sum of two cones over $X$ at their vertices  and $\tilde{x}_t= [(x_0, t)]$ for $t \neq 0$. One can see that $Y$ is contractible at the common point, and hence, it is $n$-simply connected. By Theorem \ref{th2.5n}, $\mathcal{H}_n(Y, \tilde{x}_t) = \mathcal{H}_n(CX , \tilde{x}_t)$.
Also, by \cite[Theorem 2.13]{1},
$\mathcal{H}_n(Y, x_0) \cong \frac{\mathcal{H}_n(X, x_0)}{\prod^W_{k \in \mathbb{N}} \pi_n(X, x_0)}$.

Let $x_*$ be the common vertex of the two cones, then $\mathcal{H}_n(Y, x_*)$ is trivial, for $Y$ is semilocally strongly contractible at $x_*$ and by Theorem \ref{th2.6n}, $\mathcal{H}_n(Y, x_*) = \prod_{\mathbb{N}}^W \pi_n(Y, x_*)$ which is trivial.
\end{example}

The following example reveals that Theorem \ref{th2.5n} does not hold, if some non small loop exists.

\begin{example}
Let $(X, x_*)= (C(\mathbb{HE}^1),\theta) \vee (\mathbb{S}^1, 1)$ (see Figure \ref{fi2}) .
If one assumes that $\mathcal{H}_1(X, x_*) = \mathcal{H}_1(C(\mathbb{HE}^1), x_*)$, then the simple $1$-loop in $\mathbb{S}^1$ must be nullhomotopic which is a contradiction.

\begin{figure}[!ht]
\centering
\includegraphics[scale=.5]{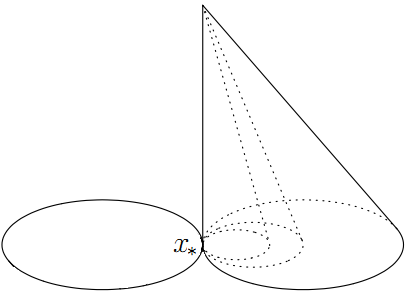}
\caption{The wedge sum of a circle and a cone on the Hawaiian earring}
\label{fi2}
\end{figure}
\end{example}

\section{Hawaiian Groups of Griffiths Spaces}
In this section, by generalizing the Griffiths space to higher dimensions and applying the results of Sections 2 and 3, we study the $n$th Hawaiian group of the $n$th Griffiths space, for $n \ge 1$.

Eda \cite{Eda} introduced the free $\sigma$-product $\times^{\sigma}_{\mathbb{N}} \mathbb{Z}$ as the group consisting of all reduced $\sigma$-words, and then proved that it is isomorphic to $\pi_1(\mathbb{HE}^1, \theta)$ \cite[Theorem A.1]{Eda}. To prove, Eda remarked that each $1$-loop in the $1$-dimensional Hawaiian earring is homotopic to some proper $1$-loop \cite[Lemma A.3]{Eda}. A $1$-loop $\alpha: (\mathbb{I}, \dot{\mathbb{I}}) \to (\mathbb{HE}^1, \theta)$ is called proper whenever for each subinterval $[a,b] \subseteq \mathbb{I}$, if $\alpha|_{[a,b]}$ is nullhomotopic, then it is constant.

Also, O. Bogopolski  and A. Zastrow  proved  that
\[
\pi_1(\mathcal{G}_1, a) \cong \frac{\times^{\sigma}_{\mathbb{N}} \mathbb{Z}}{\overline{\langle \times^{\sigma}_{\mathbb{N}_e} \mathbb{Z}, \times^{\sigma}_{\mathbb{N}_o} \mathbb{Z}\rangle}^N},
\]
where $\mathbb{N}_o$ and $\mathbb{N}_e$ denote the set of odd and even numbers, respectively \cite[Theorem 3.4]{BogZas}.  This isomorphism is induced by the natural embedding $\iota:\mathbb{HE}^1 \to \mathcal{G}_1$ causing epimorphism $\pi_1(\iota): \times^{\sigma}_{\mathbb{N}} \mathbb{Z} \to \pi_1(\mathcal{G}_1, x_*)$ together with $\overline{\langle \times^{\sigma}_{\mathbb{N}_e} \mathbb{Z}, \times^{\sigma}_{\mathbb{N}_o} \mathbb{Z}\rangle}^N$ as its kernel.

The following lemma gives a useful description of the group $\mathcal{H}_1(\mathbb{HE}^1, \theta)$ which is used in sequel.

\begin{lemma}\label{le4.1nn}
Let $\mathcal{B}$ be the subgroup of $\prod_{\mathbb{N}} \times^{\sigma}_{\mathbb{N}} \mathbb{Z}$ consisting of all countably infinite tuples of reduced $\sigma$-words that the number of components including letter of type $m$ is finite, for all $m \in \mathbb{N}$. Then
\[
\mathcal{H}_1(\mathbb{HE}^1, \theta) \cong \mathcal{B}.
\]
\end{lemma}

\begin{proof}
By \cite[Theorem 2.9]{1}, $\mathcal{H}_1(\mathbb{HE}^1, \theta)$ is isomorphic to the subgroup of $\prod_{\mathbb{N}} \pi_1(\mathbb{HE}^1, \theta)$ consisting of all sequences of homotopy classes of $1$-loops with some representative converging uniformly to the constant $1$-loop. We show that $\mathcal{B}$ equals this subgroup.

Let $\{U_l;\ l \in \mathbb{N}\}$ be the local basis at $\theta$ defined in the proof of \cite[Theorem 2.9]{1}, and also let $\{[f_k]\} \in \prod_{\mathbb{N}} \pi_1(\mathbb{HE}^1, \theta)$.  Then $\{f_k\}$ converges uniformly to the constant $1$-loop if and only if  for each $l \in \mathbb{N}$, there exists $K_l \in \mathbb{N}$ such that $im(f_k) \subseteq U_l$ whenever $k \ge K_l$. Assume that $f_k$ is the corresponding proper representative for all $k \in \mathbb{N}$. The image of $f_k$ is contained in $U_l$ if and only if $im(f_k)  \cap \mathbb{S}_m^1= \{\theta\}$, where $(m <l)$, for $U \cap \mathbb{S}^1_m$ is contractible and $f_k$ has no trivial subpath. Therefore, $\{f_k\}$ converges uniformly to the constant $1$-loop if and only if  for each $l \in \mathbb{N}$, there exists $K_l \in \mathbb{N}$ such that $im(f_k)  \cap \mathbb{S}_m^1= \{\theta\}$ whenever $k \ge K_l$.

By \cite[Theorem A.1]{Eda}, $\pi_1(\mathbb{HE}^1, \theta) \cong  \times^{\sigma}_{\mathbb{N}} \mathbb{Z}$, the group of reduced $\sigma$-words. Moreover, in a given reduced $\sigma$-word, the letter of type $m$ exists if and only if the $m$th circle $\mathbb{S}_m^1$ of $\mathbb{HE}^1$ appears in the corresponding proper $1$-loop. Therefore $\mathcal{H}_1(\mathbb{HE}^1, \theta)$ is isomorphic to the subgroup of $\prod_{\mathbb{N}} \times^{\sigma}_{\mathbb{N}} \mathbb{Z}$  consisting of all countably infinite tuples of reduced $\sigma$-words such that the number of components including letter of type $m$ is finite, for all $m \in \mathbb{N}$.
\end{proof}

The following theorem investigates the structure of the $1$th Hawaiian group of the $1$th Griffiths space  at the common point, the two vertices and the other points. Let $i:\mathbb{HE}^1  \to \mathcal{G}_1 \setminus \{v_1, v_2\}$ be the embedding which maps $(2m-1)$th circle onto the the horizontal left $m$th circle and maps $2m$th circle onto the horizontal right $m$th  circle, for $m \in \mathbb{N}$ (see Figure \ref{fi1}).  Also, let $j: \mathcal{G}_1 \setminus \{v_1, v_2\} \to \mathcal{G}$ be the inclusion map . 

\begin{theorem}\label{th4.3n}
Let $\mathcal{G}_1$ be the $1$th Griffiths space, $x_*$ the common point, $a \in A \cup A'$, and $x \in \mathcal{G}_1 \setminus (A \cup A' \cup \{x_*\})$. Then
\begin{align}
& \mathrm{(i)}& \mathcal{H}_1(\mathcal{G}_1, x_*) & \cong \frac{\mathcal{H}_1(\mathbb{HE}^1, \theta)}{i_*^{-1} \ker j_*} \cong \frac{\mathcal{B}}{\prod_{\mathbb{N}}^W\overline{\langle \times^{\sigma}_{\mathbb{N}_e} \mathbb{Z}, \times^{\sigma}_{\mathbb{N}_o} \mathbb{Z}\rangle}^N}, \label{eq1.2}
\\
&\mathrm{(ii)}& \mathcal{H}_1(\mathcal{G}_1, a)
& \cong \frac{{\mathcal{H}}_1(\mathbb{HE}^1,\theta)}{\prod_{\mathbb{N}}^W{\pi}_1(\mathbb{HE}^1,\theta)} \ltimes {\prod_{\mathbb{N}}}^W  \pi_1(\mathcal{G}_1, a) \cong
\frac{\mathcal{B}}{\prod^W_{\mathbb{N}} \times^{\sigma}_{\mathbb{N}} \mathbb{Z}} \ltimes {\prod_{\mathbb{N}}}^W  \frac{\times^{\sigma}_{\mathbb{N}} \mathbb{Z}}{\overline{\langle \times^{\sigma}_{\mathbb{N}_e} \mathbb{Z}, \times^{\sigma}_{\mathbb{N}_o} \mathbb{Z}\rangle}^N}, \label{eq3.1}
\\
&\mathrm{(iii)}& \mathcal{H}_1( \mathcal{G}_1 , x) &\cong {\prod_{\mathbb{N}}}^W \frac{\times^{\sigma}_{\mathbb{N}} \mathbb{Z}}{\overline{\langle \times^{\sigma}_{\mathbb{N}_e} \mathbb{Z}, \times^{\sigma}_{\mathbb{N}_o} \mathbb{Z}\rangle}^N}.\label{eq4.10n}
\end{align}
\end{theorem}

\begin{proof}
\begin{enumerate}[(i)]
\item
 Since all $1$-loops in $\mathcal{G}_1$ at $x_*$ are small, Theorem \ref{th2.5n} implies that $\mathcal{H}_1 (\mathcal{G}_1 , x_*) = j_* \mathcal{H}_1 (U_1 \vee U_2, x_*)$ when $U_1$ and $U_2$ are arbitrary neighbourhoods of $x_*$ in the two cones. Suppose that $U_m$ $(m=1,2)$ is the whole of the corresponding cone except its vertex. Then $U_1 \vee U_2 = \mathcal{G}_1 \setminus \{v_1, v_2\}$ and thus
 \[\mathcal{H}_1(\mathcal{G}_1, x_*) = j_* \mathcal{H}_1 (U_1 \vee U_2, x_*) = j_* \mathcal{H}_1 (\mathcal{G}_1 \setminus \{v_1, v_2\}, x_*).
 \]
Moreover, the embedded Hawaiian earring $i(\mathbb{HE}^1)$ is a deformation retract of $\mathcal{G}_1 \setminus \{v_1, v_2\}$ with projection $p: \mathcal{G}_1 \setminus \{v_1, v_2\} \to \mathbb{HE}^1$ as the retraction. Therefore,
\[ \mathcal{H}_1(\mathcal{G}_1, x_*) = j_* \mathcal{H}_1 (\mathcal{G}_1 \setminus \{v_1, v_2\}, x_*) \cong j_* i_*\mathcal{H}_1 (\mathbb{HE}^1, \theta).
\]
By the First Isomorphism Theorem,
\[
 j_* i_*\mathcal{H}_1 (\mathbb{HE}^1, \theta) \cong \frac{i_*\mathcal{H}_1(\mathbb{HE}^1, \theta)}{\ker j_*} \cong \frac{\mathcal{H}_1(\mathbb{HE}^1, \theta)}{i_*^{-1} \ker j_*}. \]
In the following, we show that $i_*^{-1}\ker j_*$ is mapped isomorphically onto $\prod_{\mathbb{N}}^W\overline{\langle \times^{\sigma}_{\mathbb{N}_e} \mathbb{Z}, \times^{\sigma}_{\mathbb{N}_o} \mathbb{Z}\rangle}^N$, by the same isomorphism which maps $\mathcal{H}_1(\mathbb{HE}^1, \theta)$ onto $\mathcal{B}$, in Lemma \ref{le4.1nn}.

Let $[g] \in i_*^{-1}\ker j_*$. Then $i_* [g] \in \ker j_*$, and therefore, $j \circ i \circ g \simeq C_{x_*}$. Hence, the $1$-loops $(j \circ i \circ g|_{S_k^1})$'s are nullhomotopic with some null convergent sequence of homotopies, say $\{H_k\}$. Thus, there exists $K\in \mathbb{N}$ such that $im H_k \subseteq \mathcal{G}_1 \setminus \{v_1, v_2\}$ for $k \geq K$. Therefore, $i \circ g|_{\mathbb{S}^1_k}$ is nullhomotopic in $\mathcal{G}_1 \setminus \{v_1, v_2\}$ for $k \geq K$ by null convergent homotopies $\{H_k\}_{k \ge K}$. Accordingly, $i \circ g|_{\widetilde{\bigvee}_{k \ge K}\mathbb{S}^1_k}$ is nullhomotpic in $\mathcal{G}\setminus \{v_1, v_2\}$. Therefore, $p \circ i \circ g|_{\widetilde{\bigvee}_{k \ge K}\mathbb{S}^1_k} = g|_{\widetilde{\bigvee}_{k \ge K}\mathbb{S}^1_k}$ is nullhomotopic in $\mathbb{HE}^1$.

Moreover, for $k <K$,  $j \circ i \circ g|_{S^1_k}$ is nullhomotopic in $\mathcal{G}_1$ or equivalently $[g|_{S_k^1}] \in \ker \pi_1(\iota)$. Note that $\ker \pi_1(\iota)$ is mapped isomorphically onto $\overline{\langle \times^{\sigma}_{\mathbb{N}_e} \mathbb{Z}, \times^{\sigma}_{\mathbb{N}_o} \mathbb{Z}\rangle}^N$,  by the same isomorphism mapping $\pi_1(\mathbb{HE}, \theta)$ onto $\times^{\sigma}_{\mathbb{N}} \mathbb{Z}$. Thus, $[g]$ is corresponded injectively to an element of $\prod_{\mathbb{N}}^W\overline{\langle \times^{\sigma}_{\mathbb{N}_e} \mathbb{Z}, \times^{\sigma}_{\mathbb{N}_o} \mathbb{Z}\rangle}^N$, where the corresponding is the same as the isomorphism mapping $\mathcal{H}_1(\mathbb{HE}, \theta)$ onto $\mathcal{B}$, in Lemma \ref{le4.1nn}. One can check that this correspondence is also surjective.

\item
By Theorem \ref{th4.2n},
\[
\mathcal{H}_1(\mathcal{G}_1, a) \cong
\mathcal{H}_1(C\mathbb{HE}^1, a) \ltimes {\prod_{\mathbb{N}}}^W  \pi_1(\mathcal{G}_1, a) \cong \mathcal{H}_1(C\mathbb{HE}^1, a) \ltimes {\prod_{\mathbb{N}}}^W  \frac{\times^{\sigma}_{\mathbb{N}} \mathbb{Z}}{\overline{\langle \times^{\sigma}_{\mathbb{N}_e} \mathbb{Z}, \times^{\sigma}_{\mathbb{N}_o} \mathbb{Z}\rangle}^N}.
\]
Now by \cite[Theorem 2.13]{1}
$\mathcal{H}_1(C\mathbb{HE}^1, a) \cong \frac{{\mathcal{H}}_1(\mathbb{HE}^1,\theta)}{\prod_{\mathbb{N}}^W{\pi}_1(\mathbb{HE}^1,\theta)}$.

By Lemma \ref{le4.1nn}, $ \mathcal{H}_1(\mathbb{HE}^1, \theta) \cong \mathcal{B}$. This isomorphism maps the subgroup $\prod_{\mathbb{N}}^W {\pi}_1(\mathbb{HE}^1,\theta)$ onto $\prod^W_{\mathbb{N}}\times^{\sigma}_{\mathbb{N}} \mathbb{Z}$. Consequently
$\mathcal{H}_1( C\mathbb{HE}^1, a) \cong  \frac{\mathcal{B}}{\prod^W_{\mathbb{N}} \times^{\sigma}_{\mathbb{N}} \mathbb{Z}}$,
and hence the isomorphism \eqref{eq3.1} holds.

\item
Obviously, $\mathcal{G}_1$ is semilocally strongly contractible at $x$. Therefore, $\mathcal{H}_1 (\mathcal{G}_1, x) \cong \prod_{\mathbb{N}}^W \pi_1(\mathcal{G}_1, x)$.
\end{enumerate}
\end{proof}

We define the $n$th Griffiths space by the wedge sum of two copies of cones on $\mathbb{HE}^n$ at the origin for $n\ge 2$. In the following theorem we give some information on the structure of the $n$th Hawaiian group of the $n$th Griffiths space.
Note that by $\mathbb{Z}^k$ we mean  the subgroup of $\bigoplus_{\mathbb{N}} \mathbb{Z}$ consisting of all countably infinite tuples with all zero components except the first $k$ components and thus $\mathbb{Z}^1 \leqslant \mathbb{Z}^2 \leqslant \mathbb{Z}^3 \leqslant \cdots$.

\begin{theorem}\label{th3.2n}
Let $n \ge 2$ and $\mathcal{G}_n$ be the $n$th Griffiths space, $a \in A \cup A'$, and $x \in \mathcal{G}_1 \setminus (A \cup A' \cup \{x_*\})$. Then
\begin{align*}
&\mathrm{(i)}&
 \mathcal{H}_n(\mathcal{G}_n, a) &\cong
\frac{\mathcal{H}_n(\mathbb{HE}^n, \theta)}{\bigoplus_{\mathbb{N}} \pi_n(\mathbb{HE}^n, \theta)}  \oplus \bigoplus_{\mathbb{N}} \pi_n(\mathcal{G}_n) \cong
\frac{\prod_{\mathbb{N}} \bigoplus_{\mathbb{N}} \mathbb{Z}}{\bigcup_{k \in \mathbb{N}} \prod_{\mathbb{N}} \mathbb{Z}^k}    \oplus \bigoplus_{\mathbb{N}} \pi_n(\mathcal{G}_n) 
\cr
& \mathrm{(ii)}&
\mathcal{H}_n( \mathcal{G}_n , x) &\cong \bigoplus_{\mathbb{N}} \pi_n(\mathcal{G}_n) \label{eq2.9}.
\end{align*}
\end{theorem}
\begin{proof}
\begin{enumerate}[(i)]
\item
By Theorem \ref{th4.4},
$\mathcal{H}_n(\mathcal{G}_n, a) \cong
\mathcal{H}_n(C\mathbb{HE}^n, a) \oplus \bigoplus_{\mathbb{N}} \pi_n(\mathcal{G}_n, a)$.
Using \cite[Theorem 2.13]{1},
$\mathcal{H}_n(C\mathbb{HE}^n, a) \cong \frac{{\mathcal{H}}_n(\mathbb{HE}^n,\theta)}{\bigoplus_{\mathbb{N}} \pi_n(\mathbb{HE}^n , \theta)}$. Replacing  $\mathcal{H}_n(\mathbb{HE}^n, \theta)$ with $\prod_{\mathbb{N}}\bigoplus_{\mathbb{N}} \mathbb{Z}$ via \cite[Theorem 2.11]{1}, which maps $\bigoplus_{\mathbb{N}} \pi_n(\mathbb{HE}^n , \theta)$ isomorphically onto $\bigcup_{k \in \mathbb{N}} \prod_{\mathbb{N}} \mathbb{Z}^k$, the result holds.
\item
Since $\mathcal{G}_n$ is semilocally strongly contractible at $x$, the isomorphism holds by Theorem \ref{th2.6n}.
\end{enumerate}
\end{proof}

We recall that the $n$th Hawaiian group of the $n$th Griffiths space at the common point is generated by two its subgroups;  $\mathcal{H}_n(\mathcal{G}_n, x_*) =  j_* \mathcal{H}_n( U_1 \vee U_2, x_*)   \big(\bigoplus_{\mathbb{N}} \pi_n(\mathcal{G}_n,x_*)\big)$, when $U_1$ and $U_2$ are arbitrary neighbourhoods of the origin in two cones, and $j:U_1 \vee U_2 \to \mathcal{G}_n$ is the inclusion map.

\section*{Acknowledgements}
This research was supported by a grant from Ferdowsi University of Mashhad-Graduate Studies (No. 42705).


\end{document}